\documentclass{amsart}
\usepackage{amsmath,mathrsfs}
\usepackage{amsfonts}
\usepackage{amssymb,color}
\usepackage{graphicx}
\usepackage[square,numbers]{natbib}
\usepackage{verbatim}
\usepackage{enumerate}
\newtheorem{theorem}{Theorem}
\newtheorem{claim}[theorem]{Claim}
\newtheorem{corollary}[theorem]{Corollary}
\newtheorem{lemma}[theorem]{Lemma}
\newtheorem{proposition}[theorem]{Proposition}

\theoremstyle{definition}
\newtheorem{definition}[theorem]{Definition}

\theoremstyle{remark}
\newtheorem{remark}[theorem]{Remark}

\numberwithin{theorem}{section}
\numberwithin{equation}{section}

\def\XXint#1#2#3{{\setbox0=\hbox{$#1{#2#3}{\int}$}
     \vcenter{\hbox{$#2#3$}}\kern-.5\wd0}}

\newcommand{\bbH}{\mathbb{H}}
\newcommand{\bbR}{\mathbb{R}}
\newcommand{\bbN}{\mathbb{N}}

\begin{document}
\title[Whitney extension for curves in the Heisenberg group]{A $C^m$ Whitney Extension Theorem for Horizontal Curves in the Heisenberg Group}

\author[Andrea Pinamonti]{Andrea Pinamonti}
\address[Andrea Pinamonti]{Department of Mathematics, University of Trento, Via Sommarive 14, 38123 Povo (Trento), Italy}
\email[Andrea Pinamonti]{Andrea.Pinamonti@unitn.it}

\author[Gareth Speight]{Gareth Speight}
\address[Gareth Speight]{Department of Mathematical Sciences, University of Cincinnati, 2815 Commons Way, Cincinnati, OH 45221, United States}
\email[Gareth Speight]{Gareth.Speight@uc.edu}

\author[Scott Zimmerman]{Scott Zimmerman}
\address[Scott Zimmerman]{Department of Mathematics, University of Connecticut, 341 Mansfield Road U1009, Storrs, Connecticut 06269, United States}
\email[Scott Zimmerman]{Scott.Zimmerman@uconn.edu}


\date{\today}

\begin{abstract}
We characterize those mappings from a compact subset of $\mathbb{R}$ into the Heisenberg group $\bbH^{n}$ 
which can be extended to a $C^{m}$ horizontal curve in $\mathbb{H}^{n}$. The characterization combines the classical Whitney conditions with an estimate comparing changes in the vertical coordinate with those predicted by the Taylor series of the horizontal coordinates.
\end{abstract}

\maketitle


\section{Introduction}\label{introduction}

The classical Whitney extension theorem \cite{Bie80, Whi34} characterizes those collections of continuous functions $F=(F^{k})_{k=0}^{m}$ defined on a compact set $K\subset \bbR^{n}$ which can be extended to a $C^m$ function $f$ defined on an open set containing $K$ so that the derivatives $D^kf$ of the extension coincide with the functions $F^k$ on $K$. The collection $(F^{k})_{k=0}^{m}$ is called a \emph{jet} of order $m$ (see Definition~\ref{jet}), and we intuitively view the functions $F^k$ for $k \geq 1$ as some sort of ``derivatives'' of $F^0$. A $C^m$ extension is guaranteed to exist if the jet is a so called Whitney field of class $C^m$ on $K$ (see Definition~\ref{whitneyfield}). Intuitively, this means that the value of each $F^{k}$ on $K$ should be uniformly well approximated by the Taylor polynomial centered at any nearby point $a\in K$ computed using the jet. Taylor's theorem ensures that this approximation holds for any $C^m$ function restricted to $K$, so Whitney's theorem acts as a sort of converse to Taylor's theorem.

The main result of this paper is a $C^m$ Whitney extension theorem for mappings from compact subsets of $\bbR$ into the Heisenberg group (Theorem~\ref{iff}). Before describing our result, we first give some motivation and history related to the problem.

Whitney's classical extension theorem has many applications. For instance, it can be used to construct functions with unusual differentiability properties \cite{Whi35} and to construct $C^1$ approximations of Lipschitz mappings \cite{Fed69}. Such approximations have been used to show that rectifiable sets may be equivalently defined by Lipschitz functions or by $C^1$ functions. Recently, great attention has been devoted to the study of quantitative versions of Whitney's theorem.  More specifically, given all (or just a part) of the Whitney data on $K$, one can attempt to construct a smooth extension with some sort of reasonable estimate on the $C^m$-norm. This problem is of interest in applications, and it is highly nontrivial even in the setting of functions defined at finitely many points \cite{Feff1,Feff2,Feff3,Feff4}.

 In recent years, it has become clear that a large part of geometric analysis in Euclidean spaces may be generalized to more general settings  \cite{BLU07, Che99, CDPT07, HKST15, Mon02, Pan89}. In particular, rectifiable sets are currently under intense study in Carnot groups such as the Heisenberg group \cite{FSS01, FSS03, FSS03b}. This demonstrates the importance of understanding to what extent a version of Whitney's extension theorem holds for mappings between more general spaces.

Carnot groups are Lie groups whose Lie algebra admits a stratification. This stratification gives rise to dilations and implies that points can be connected by absolutely continuous curves with tangents in a distinguished subbundle of the tangent bundle. These are the so called horizontal curves. Considering lengths of horizontal curves gives rise to the Carnot-Carath\'{e}odory distance and endows every Carnot group with a metric space structure. Moreover, every Carnot group has a natural Haar measure which respects the group translations and dilations. This plethora of structure makes the study of analysis and geometry in Carnot groups highly interesting \cite{BLU07, CDPT07, Mon02}. However, results in the Carnot setting can be very different to Euclidean ones since all such results must respect the horizontal structure of the Carnot group. The Heisenberg group is the simplest non-Euclidean Carnot group and admits an explicit representation in $\bbR^{2n+1}$ (Definition \ref{Heisenberg}) with $2n$ horizontal directions and one vertical direction.

Validity of a Whitney extension theorem in Carnot groups has received considerable attention in recent years. The best understood case is that of mappings from Carnot groups to $\bbR$. In 2001, Franchi, Serapioni, and Serra Cassano proved a $C^1$ version of the Whitney extension theorem for mappings from the Heisenberg group into $\bbR$ \cite{FSS01}. In this theorem, the jet is defined on a compact subset $K$ of the Heisenberg group and is extended to a $C^{1}_{H}$ function. That is, the derivatives of the extension in the horizontal directions exist and are continuous. In 2006, Vodop'yanov and Pupyshev proved a $C^m$ version of the Whitney extension theorem for mappings from general Carnot groups to $\bbR$ \cite{VP06}. 

The study of a Whitney extension theorem for mappings whose target is a Carnot group is even more recent. Zimmerman \cite{Zim18} established a $C^1$ version of the Whitney extension theorem for mappings from compact subsets of $\bbR$ into the Heisenberg group. More or less at the same time, Speight \cite{Spe14} independently addressed the related problem of a Lusin type approximation of absolutely continuous horizontal curves by $C^1$ horizontal curves in the Heisenberg group. Here, it was also shown that there is no such Lusin type approximation for horizontal curves in the Engel group (which is a Carnot group of step 3). Hence, one should not expect a Whitney extension theorem for mappings from compact subsets of $\bbR$ into every Carnot group. The positive results on Lusin approximation were extended to all step 2 Carnot groups in \cite{LS16}, after which \cite{JS16, SS18} extended the $C^1$ Whitney extension theorem to mappings from compact subsets of $\bbR$ to larger classes of Carnot groups and subriemannian manifolds. 

Until now, little was known about the validity of a higher order Whitney extension theorem for mappings whose target is a non-Euclidean Carnot group. In this paper, we establish such a $C^m$ Whitney extension theorem for mappings into the Heisenberg group. For simplicity we focus on the first Heisenberg group $\bbH^{1}$, represented in coordinates as $\bbR^{3}$, but similar methods work for any Heisenberg group $\bbH^{n}$. Our result characterizes when a triple of jets $(F^{k})_{k=0}^{m}$, $(G^{k})_{k=0}^{m}$, and $(H^{k})_{k=0}^{m}$ defined on a compact subset $K\subset \bbR$ can be extended to a $C^m$ horizontal curve from $\bbR$ to $\bbH^{1}$. 

We say that a triple $(F,G,H)$ of jets of order $m$ \emph{extends to the $C^m$ horizontal curve $(f,g,h)\colon \bbR\to \bbH^{1}$} if $(f,g,h)$ is a $C^m$ horizontal curve from $\bbR$ into $\bbH^{1}$ and we have $D^kf(x)=F^{k}(x)$, $D^kg(x)=G^{k}(x)$ and $D^kh(x)=H^{k}(x)$ for all $x\in K$ and $0\leq k\leq m$. The first two conditions of our characterization (Theorem~\ref{iff}) are easy to understand. First, $f$, $g$, and $h$ must extend $F$, $G$, and $H$ as maps into $\bbR$. Hence, by Taylor's theorem, the jets $F$, $G$, and $H$ must already be Whitney fields of class $C^m$ on $K$. Secondly, differentiating in the definition of a horizontal curve gives \eqref{HigherHoriz} where polynomials $\mathcal{P}^k$ establish a relationship between the derivatives of the different components of the curve. However, these two conditions alone are not enough. This is shown in Proposition \ref{propexample}. The problem arises from the fact that the vertical component of a horizontal curve is not free to vary but is instead determined by an area swept out by the horizontal components (see Lemma~\ref{lift}). 

Presumably, it is redundant to assume that $H$ is a Whitney field in Theorem~\ref{iff}(1) 
when proving sufficiency of the assumptions.
This property is likely a result of assumptions (2) and (3) in Theorem~\ref{iff}
together with the Whitney properties of $F$ and $G$.
(For an easy proof of this fact in the case $m=1$, see Remark 1.6 in \cite{Zim18}.)
However, since the Whitney field property of $H$ is necessary for the existence of a $C^m$ extension, 
it is reasonable to include this condition.

We now describe the third condition in our characterization and state the result.
Given $a,b\in K$, define the \emph{area discrepancy}
\begin{align}\label{Aab}
 A(a,b)&:=H(b)-H(a)-2\int_{a}^{b}((T_{a}^mF)'(T_{a}^mG)-(T_{a}^mG)'(T_{a}^mF))\\
\nonumber &\qquad +2F(a)(G(b)-T_{a}^mG(b))-2G(a)(F(b)-T_{a}^mF(b)).
\end{align}
Here we use the identification $F(x) = F^0(x)$, $G(x) = G^0(x)$, and $H(x) = H^0(x)$.
The terms $T_{a}^mF$ and $T_{a}^mG$ denote the Taylor polynomials of the jets $F$ and $G$ 
(see Definition \ref{taylor}). Note that $A(a,b)$ measures over $[a,b]$ the difference between the change in height of the jets and the change in height predicted by lifting the Taylor expansion of the horizontal components. The terms on the second line of \eqref{Aab} are a result of the group operation in $\mathbb{H}^1$ when we consider points away from the origin. Intuitively, in order for a $C^m$ horizontal extension to exist, the area discrepancy $A(a,b)$ must be very small as $(b-a)\to 0$. To make this precise, we define the \emph{velocity}
\begin{equation}\label{Vab}
V(a,b):= (b-a)^{2m} + (b-a)^{m} \int_{a}^{b}|(T_{a}^mF)'|+|(T_{a}^mG)'|.
\end{equation}
In some sense, $V(a,b)$ is related to the speed of the horizontal components of the curve fragment. If the higher order terms in the jets $F$ and $G$ are large at $a$, then $V(a,b)$ is controlled by the integral term. Otherwise $V(a,b)$ is controlled by $(b-a)^{2m}$. The final condition of our characterization asserts $A(a,b)/V(a,b)\to 0$ uniformly as $(b-a) \searrow 0$ with $a,b\in K$.

We now state formally our main theorem.

\begin{theorem}\label{iff}
Let $K\subset \bbR$ be a compact set and $F$, $G$, and $H$ be jets of order $m$ on $K$. 
Then the triple $(F,G,H)$ extends to a $C^m$ horizontal curve $(f,g,h)\colon \bbR \to \bbH^{1}$ if and only if
\begin{enumerate}
\item $F$, $G$, and $H$ are Whitney fields of class $C^m$ on $K$,
\item for every $1 \leq k \leq m$ and $t\in K$ we have
 \begin{equation}
\label{HorizAssume}
H^{k}(t) = \mathcal{P}^k \left(F^0(t),G^0(t),F^1(t),G^1(t),\dots,F^{k}(t),G^{k}(t)\right),
\end{equation}
\item $A(a,b)/V(a,b)\to 0$ uniformly as $(b-a) \searrow 0$ with $a,b\in K$.
\end{enumerate}
\end{theorem}

First note that Theorem \ref{iff} is consistent with the $C^1$ case from \cite{Zim18}, even if condition (3) appears slightly different to the corresponding condition in \cite{Zim18}. This is described in Section \ref{C1}. Our proof of Theorem \ref{iff} when specialized to the $C^1$ case is very different to that of \cite{Zim18}. Hence we also obtain a new proof of the $C^1$ case. We now briefly describe the structure of the proof.

Necessity of the three conditions in Theorem \ref{iff} is established in Proposition \ref{nec}. As already described, (1) follows from Taylor's theorem, and (2) follows from differentiating the definition of a horizontal curve in Lemma~\ref{lift}. To establish (3), we assume $f(a)=g(a)=h(a)=0$ and combine the definition of a horizontal curve with direct estimates on $\int_{a}^{b}|f'g-(Tf)'Tg|$. To remove the assumption $f(a)=g(a)=h(a)=0$, we simply apply group translations.

Sufficiency of the three conditions is more involved and is established by Theorem~\ref{newthm}. We begin with the decomposition $[\min K, \ \max K]\setminus K=\cup_{i\geq 1}(a_{i},b_{i})$. The main step will be to obtain Lemma~\ref{buildcurve}. This lemma provides $C^m$ horizontal curves $(\mathcal{F}_{i}, \mathcal{G}_{i}, \mathcal{H}_{i})$ defined on each interval $[a_{i},b_{i}]$ whose derivatives agree with the values prescribed by $(F,G,H)$ at the endpoints, the derivatives do not deviate far from these values along the entire interval, and the areas enclosed by $\mathcal{F}_i$ and $\mathcal{G}_i$ in the plane are chosen so that, when they are lifted to a horizontal curve in the Heisenberg group, the change in height $\mathcal{H}_i(b_i) - \mathcal{H}_i(a_i)$ is equal to $H(b_i) - H(a_i)$. Once these curves are constructed, one can glue such curves to obtain the required extension of $(F,G,H)$ (Proposition \ref{conclusion}). To establish Lemma \ref{buildcurve}, we begin by using the classical Whitney extension theorem to extend the jets $F$ and $G$ to $C^m$ functions $f$ and $g$ whose derivatives take the correct values on $K$. On each interval $[a_{i},b_{i}]$ we then construct perturbations $\phi$ and $\psi$ whose derivatives vanish at the endpoints and are uniformly small throughout $(a_{i},b_{i})$ so that lifting the curve $(f+\phi, g+\psi)$ creates a horizontal curve whose height agrees with $H$ at the endpoints of $[a_{i},b_{i}]$. The difficulty arises in ensuring that these new curves actually meet the image of $K$ at the correct height. Fortunately, the condition on $A(a,b)/V(a,b)$ controls the change in height over $K$ and therefore guarantees that small enough perturbations can in fact be constructed despite the constraints on them. To build perturbations large enough to meet the height requirements, inequalities such as
\[\int_{a}^{b} \psi (Tf)'\geq C\beta (b-a)^{m}\int_{a}^{b}|(Tf)'|\]
are necessary, where $\beta$ is the bound on the derivatives of $\psi$ and $C$ is some constant. To obtain this, we use Markov's inequality to control the behavior of polynomials such as $Tf$ (Lemma~\ref{Pbig}) which asserts that the largest values of a polynomial cannot be overly concentrated in too small a region.

The structure of the paper is as follows. In Section~\ref{prelim} we recall preliminaries on the Heisenberg group, the classical Whitney extension theorem, and useful inequalities for polynomials. In Section~\ref{consistency} we show that Theorem~\ref{iff} is consistent with previous results for the $C^1$ case. In Section~\ref{example} we give an example of jets satisfying conditions (1) and (2) of Theorem~\ref{iff} for which no $C^m$ horizontal extension exists. In Section~\ref{sectionnecessary} we prove the easier implication of necessity of the conditions in Theorem \ref{iff}. Finally, in Section~\ref{sectionsufficient} we prove that the conditions of Theorem~\ref{iff} are sufficient for the existence of a $C^m$ horizontal extension.

\begin{remark}
To simplify notation we have restricted our attention to the first Heisenberg group $\mathbb{H}^1$. Theorem \ref{iff} and its proof generalize to the setting of curves in an arbitrary Heisenberg group $\bbH^{n}$, which is identified with $\bbR^{2n+1}$ in coordinates. In this setting, $F_{1},\ldots, F_{n}, G_{1}, \ldots, G_{n}$, and $H$ are jets of order $m$ on $K$, and our aim is to extend $(F_{1}, \ldots, F_{n}, G_{1}, \ldots, G_{n}, H)$ to a $C^m$ horizontal curve in $\bbH^{n}$. Condition (1) of Theorem \ref{iff} is replaced by the requirement that $F_{1},\ldots, F_{n},G_{1}, \ldots, G_{n}$, and $H$ are all Whitney fields of class $C^m$ on $K$. Condition (2) is substituted with a related condition on the polynomials $\mathcal{P}^k$ resulting from differentiating the horizontality condition for a curve $\gamma$ in $\bbH^{n}$:
\begin{equation}\label{horizHn}
\gamma_{2n+1}'(t)=2\sum_{i=1}^{n} (\gamma_{i}'(t)\gamma_{n+i}(t)-\gamma_{n+i}'(t)\gamma_{i}(t)).
\end{equation}
Condition (3) takes the same form except the definitions of $A(a,b)$ and $V(a,b)$ are modified to reflect the fact that \eqref{horizHn} is a sum of $n$ terms in $\bbH^n$, each of which is very similar to the one term which appears in $\bbH^{1}$.

The proof of necessity in $\bbH^{n}$ is essentially the same as before except each commutator term in the new definition of $A(a,b)$ is estimated separately, then the estimates are added together. The proof of sufficiency is $\bbH^{n}$ also follows the same ideas. Each of the horizontal terms $F_{1}, \ldots, F_{n}$ and $G_{1},\ldots, G_{n}$ is extended separately to a $C^m$ map. The analogue of Proposition \ref{perturb} then constructs perturbations for each horizontal term so that lifting the horizontal terms gives the right boundary conditions for $H$ in Proposition \ref{perturb}(3). Actually, one only needs to perturb at most two coordinates, depending on which terms are large in the expression for $V(a,b)$. Once these interpolating maps are constructed, the remainder of the proof of Theorem \ref{iff} in $\bbH^{n}$ is essentially the same as before. 
\end{remark}

It is natural to ask whether the ideas in the present paper can be applied in more general Carnot groups. We intend to investigate this problem in the future.

\vspace{.3cm}

\textbf{Acknowledgements:} Part of this work was done while A. Pinamonti was visiting the University of Cincinnati. This visit was partly supported by a Research Support Grant from the Taft Research Center at the University of Cincinnati.
A. P. is a member of {\em Gruppo Nazionale per l'Analisi Ma\-te\-ma\-ti\-ca, la Probabilit\`a e le loro Applicazioni} (GNAMPA) of the {\em Istituto Nazionale di Alta Matematica} (INdAM).

\section{Preliminaries}\label{prelim}

\subsection{The Heisenberg group}

\begin{definition}\label{Heisenberg}
The Heisenberg group $\mathbb{H}^{n}$ is the Lie group represented in coordinates by $\mathbb{R}^{2n+1}$, whose points we denote by $(x,y,t)$ with $x,y\in \mathbb{R}^{n}$ and $t\in \mathbb{R}$. The group law is given by:
\[(x,y,t) (x',y',t')=\left(x+x',y+y',t+t'+2\sum_{i=1}^{n}(y_{i}x_{i}'-x_{i}y_{i}')\right).\]
\end{definition}

We equip $\mathbb{H}^{n}$ with left invariant vector fields
\begin{equation}\label{heisenbergvectors}
X_{i}=\partial_{x_{i}}+2y_{i}\partial_{t}, \quad Y_{i} = \partial_{y_{i}}-2x_{i}\partial_{t}, \quad 1\leq i\leq n, \quad T=\partial_{t}.
\end{equation}
Here $\partial_{x_{i}}, \partial_{y_{i}}$ and $\partial_{t}$ denote the coordinate vectors in $\mathbb{R}^{2n+1}$, which may be interpreted as operators on differentiable functions. If $[\cdot, \cdot]$ denotes the Lie bracket of vector fields, then $[X_{i}, Y_{i}]=-4T$. Thus $\mathbb{H}^{n}$ is a Carnot group with horizontal layer $\mathrm{Span}\{X_{i}, Y_{i} \colon 1\leq i\leq n\}$ and second layer $\mathrm{Span}\{T\}$.

\begin{definition}
A vector in $\mathbb{R}^{2n+1}$ is horizontal at $p \in \mathbb{R}^{2n+1}$ if it is a linear combination of the vectors $X_{i}(p), Y_{i}(p), 1\leq i\leq n$.

An absolutely continuous curve $\gamma$ in the Heisenberg group is horizontal if, at almost every point $t$, the derivative $\gamma'(t)$ is horizontal at $\gamma(t)$.
\end{definition}

\begin{lemma}\label{lift}
An absolutely continuous curve $\gamma\colon [a,b]\to \mathbb{R}^{2n+1}$ is a horizontal curve in the Heisenberg group if and only if, for $t\in [a,b]$:
\[\gamma_{2n+1}(t)=\gamma_{2n+1}(a)+2\sum_{i=1}^{n}\int_{a}^{t} (\gamma_{i}'\gamma_{n+i}-\gamma_{n+i}'\gamma_{i}).\]
\end{lemma}

The integrals above have a geometric interpretation; if the curve starts at $0$ and is smooth enough to apply Stokes' theorem, then each gives a signed area in $\mathbb{R}^2$. 

\begin{lemma}\label{area}
Suppose $\sigma\colon [a,b]\to \mathbb{R}^{2}$ is a smooth curve with $\sigma(a)=0$. Let $[0,\sigma(b)]$ be the straight line from $0$ to $\sigma(b)$ and let $A_{\sigma}$ denote the signed area of the region enclosed by $\sigma$ and $[0,\sigma(b)]$. Then:
\[A_{\sigma}=\frac{1}{2}\int_{a}^{b}(\sigma_{1}\sigma_{2}'-\sigma_{2}\sigma_{1}').\]
\end{lemma}

Clearly Lemma \ref{lift} implies that for any horizontal curve $\gamma$ we have
$$
\gamma_{2n+1}'(t) = 2\sum_{i=1}^{n} (\gamma_{i}'(t)\gamma_{n+i}(t)-\gamma_{n+i}'(t)\gamma_{i}(t)) 
\quad
\text{for a.e. } t \in [a,b].
$$ 
If we assume that $\gamma$ is $C^1$,
this equality holds for every $t \in [a,b]$.
If we further assume that $\gamma$ is $C^m$ for some $m > 1$,
then, 
for $1 \leq k \leq m$, 
we may write 
\begin{equation}
\label{HigherHoriz}
D^k\gamma_{2n+1}(t) = \sum_{i=1}^{n} \mathcal{P}^k\left(\gamma_{i}(t),\gamma_{n+i}(t),\gamma_{i}'(t),\gamma_{n+i}'(t),\dots,D^k\gamma_{i}(t),D^k\gamma_{n+i}(t)\right)
\end{equation}
for all $t\in [a,b]$ where $\mathcal{P}^k$ is a polynomial determined by the Leibniz rule.

\subsection{Jets and the classical Whitney extension theorem}

\begin{definition}\label{jet}
A \emph{jet} of order $m$ on a set $K\subset \mathbb{R}$ consists of a collection 
of continuous functions $F=(F^{k})_{k=0}^{m}$ on $K$. We denote the space of such jets by $J^{m}(K)$ and write $F(x)=F^{0}(x)$ for $x\in K$.
\end{definition}

For an open set $U \subset \mathbb{R}$,
we define the mapping $J^m:C^m(U) \to J^m(U)$, which sends a $C^m$ function on $U$ to the jet on $U$
consisting of derivatives up to order $m$,
as $J^m(F) = (D^k F)_{k=0}^m$ for $F \in C^m(U)$. Here, $D^k$ is the $k$th derivative of $F$. 

\begin{definition}\label{taylor}
Given $a\in K$ and $F \in J^{m}(K)$, the \emph{Taylor polynomial of order $m$ of $F$ at $a$} is
\[T_{a}^{m}F(x)=\sum_{k=0}^{m} \frac{F^{k}(a)}{k!}(x-a)^{k}
\quad
\text{for all } x \in \mathbb{R}.\]
If $m$ or $a$ are clear, we write $T_{a}F$ or even $TF$ for the Taylor polynomial.
\end{definition}

Jets of the same order are added and subtracted term by term. For jets $F\in J^{m}(K)$, we will sometimes use the notation $R_{a}^{m}F=F-J^{m}(T_{a}^{m}F)$. 
This gives for $0\leq k\leq m$
\[(R_{a}^{m}F)^{k}(x)=F^{k}(x)-\sum_{\ell=0}^{m-k}\frac{F^{k+\ell}(a)}{\ell!}(x-a)^{\ell}
\quad
\text{for all } x \in \mathbb{R}.\]

\begin{definition}\label{whitneyfield}
A jet $F\in J^{m}(K)$ is a \emph{Whitney field of class $C^m$ on $K$} if, for every $0\leq k\leq m$, we have
\[(R_{a}^{m}F)^{k}(b)=o(|a-b|^{m-k})\]
as $|a-b|\to 0$ with $a,b \in K$.
\end{definition}

The following theorem is the classical Whitney extension theorem \cite{Whi34}.

\begin{theorem}[Classical Whitney extension theorem]\label{classicalWhitney}
Let $K$ be a closed subset of an open set $U\subset \mathbb{R}$. 

Then there is a continuous linear mapping $W$ from the space of Whitney fields of class $C^m$ on $K$ to $C^m(U)$ such that
\[D^{k}(WF)(x)=F^{k}(x) \quad \mbox{for $0\leq k\leq m$ and $x\in K$},\]
and $WF$ is $C^{\infty}$ on $U\setminus K$.
\end{theorem}

In other words, 
given a Whitney field $F=(F^k)_{k=0}^m$ of class $C^m$ on $K$,
there is a function $f = WF \in C^m(U)$ so that $D^kf = F^k$ on $K$ for $0\leq k\leq m$.
We now record two consequences of the proof of Theorem \ref{classicalWhitney} from \cite[p150]{Bie80} which will be useful later. 
(See also \cite{Folland}.)
Let $K\subset \mathbb{R}$ be a compact set.

\begin{enumerate}
\item
In this paper, we will define a \emph{modulus of continuity} to be
an increasing function $\alpha\colon [0,\infty)\to [0,\infty)$ with $\alpha(0)=0$ and $\alpha(t)\to 0$ as $t\searrow 0$. 
For any $C^m$ Whitney field $F$ on $K$, there exists a modulus of continuity $\alpha$ such that
\begin{equation}\label{modulus1}
|(R_{a}^{m}F)^{k}(x)|\leq \alpha(|x-a|)|x-a|^{m-k}
\end{equation}
for all $a, x\in K$ and $0\leq k\leq m$. 
\item
Let $U$ be a bounded open set containing $K$ and $f=WF$ be the Whitney extension constructed in the proof of Theorem \ref{classicalWhitney} in \cite{Bie80}. Then there exists a constant $C$ such that
\begin{equation}\label{modulus2}
|D^kf(x)-D^{k}(T_{a}^{m}F)(x)| \leq C \alpha(|x-a|)|x-a|^{m-k}
\end{equation}
for all $a\in K$, $x\in U$, and $0\leq k \leq m$.
\end{enumerate}

\subsection{Inequalities for polynomials}

We recall Markov's inequality for polynomials and prove some elementary consequences \cite{Mar89, Sha04}.


\begin{lemma}[Markov Inequality]\label{Markov}
Let $P$ be a polynomial of degree $n$ and $a < b$. Then
\[ \max_{[a,b]}|P'|\leq \frac{2n^{2}}{b-a} \max_{[a,b]}|P|.\]
\end{lemma}


\begin{lemma}\label{Pbig}
Let $P$ be a polynomial of degree $n$ and $a<b$. Let $M=\max_{[a,b]}|P|$. Then there exists a closed subinterval $I\subset [a,b]$ of length at least $(b-a)/4n^2$ such that $|P(x)|\geq M/2$ for all $x\in I$.
\end{lemma}

\begin{proof}
Suppose $x_0 \in [a,b]$ satisfies $P(x_0) = M$.
Set $I \subset [a,b]$ to be an interval of length $(b-a)/4n^2$ with $x_0$ as an endpoint. Without loss of generality, write $I=[x_0,x_1]$.
Suppose $|P(y)|<M/2$ for some $y \in I$.
Then the Markov inequality gives
$$
\frac{M}{2} 
< |P(y) - P(x_0)|
\leq \int_{x_0}^y |P'|
\leq \frac{b-a}{4n^2}\max_{[a,b]}|P'|
\leq \frac12 \max_{[a,b]}|P| = \frac{M}{2}
$$
which is impossible, so no such $y$ exists.
\end{proof}

\begin{corollary}\label{intmax}
Let $P$ be a polynomial of degree $n$ and $a<b$. Let $M=\max_{[a,b]}|P|$. Then
\[ \frac{M(b-a)}{8n^2}\leq \int_{a}^{b} |P|\leq M(b-a).\]
\end{corollary}

\section{Consistency with the $C^1$ case}\label{C1}\label{consistency}

In this section, we will see that, in the case $m=1$, Theorem~\ref{iff}
is consistent with the $C^1$ Whitney extension theorem for horizontal curves in the Heisenberg group
proven in \cite{Zim18}. We state \cite[Theorem 1.5]{Zim18} in our language here for convenience:
\begin{theorem}[Zimmerman]
\label{iffC1}
Let $K \subset \mathbb{R}$ be a compact set and $F$, $G$, and $H$ be jets of order 1 on $K$.
Then the triple $(F,G,H)$ extends to a $C^1$ horizontal curve $(f,g,h):\mathbb{R} \to \mathbb{H}^1$
if and only if
\begin{enumerate}
\item $F$, $G$, and $H$ are Whitney fields of class $C^1$ on $K$,
\item for every $t \in K$, we have $H^1(t) = 2(F^1(t) G(t) - F(t) G^1(t))$,
\item the following convergence is uniform as $(b-a) \searrow 0$ for $a,b \in K$:
\begin{equation}
\label{C1assume}
\frac{H(b) - H(a) - 2(F(b)G(a) - F(a)G(b))}{(b-a)^2} \to 0.
\end{equation}
\end{enumerate}
\end{theorem}

As before, we have identified the functions $F(x) = F^0(x)$, $G(x) = G^0(x)$, and $H(x) = H^0(x)$ for each $x \in K$.
Note that conditions (1) and (2) here are the same as those in Theorem~\ref{iff}.
To prove that Theorem~\ref{iff} is indeed a generalization of Theorem~\ref{iffC1}, 
we need only show that the convergence in \eqref{C1assume}
is equivalent to the uniform convergence $A(a,b)/V(a,b) \to 0$.
This will follow from the definitions of $A$ and $V$.
Indeed, in the case $m=1$, the area discrepancy \eqref{Aab} is
\begin{align*}
A(a,b) &= H(b) - H(a) \\
&\hspace{.2in} -2 \int_a^b[F^1(a)(G(a) + G^1(a)(t-a)) - G^1(a)(F(a) + F^1(a)(t-a))] \, dt \\
&\hspace{.2in} +2F(a)(G(b) - (G(a) + G^1(a)(b-a))) \\
&\hspace{.2in} - 2G(a)(F(b) - (F(a) + F^1(a)(b-a))) \\
&= H(b) - H(a) - 2(F(b)G(a) - F(a)G(b)).
\end{align*}
That is, $A(a,b)$ is nothing more than the top of the fraction in \eqref{C1assume}.
Moreover, the velocity is
$$
V(a,b) 
= (b-a)^{2} + (b-a) \int_{a}^{b}|F^1(a)|+|G^1(a)|
= (b-a)^2 (1 + |F^1(a)|+|G^1(a)|).
$$
Since $F'$ and $G'$ are continuous on the compact set $K$,
there is a uniform bound of $1 + |F^1(a)|+|G^1(a)|$.
Therefore, the uniform convergence $A(a,b)/V(a,b) \to 0$ is equivalent to the uniform convergence \eqref{C1assume},
and Theorem~\ref{iff} is consistent with \cite[Theorem 1.5]{Zim18} when $m=1$.

\section{Importance of the area condition}\label{example}

Here we will see the importance of the area-velocity assumption (3) in Theorem~\ref{iff}.
In other words, we will construct Whitney fields $F$, $G$, and $H$
on a compact set $K \subset \mathbb{R}$
which satisfy condition (2) of Theorem~\ref{iff},
but there will be no $C^m$ horizontal extension of $(F,G,H)$ on $\mathbb{R}$.
Note that, since $F$, $G$, and $H$ are all Whitney fields, the classical Whitney extension theorem
guarantees that a $C^m$ extension will exist.
However, we will show that any such extension cannot possibly be horizontal.
Indeed, such a horizontal, smooth extension would have to satisfy condition (3) everywhere
(according to Proposition~\ref{nec}), 
but our mapping will not satisfy this on $K$.

This construction will be nearly identical to \cite[Proposition 1.3]{Zim18}
which provided an example in the case $m=1$.
Here, we will show that almost the same construction works for any $m \geq 1$.
We summarize this discussion in the following proposition. 

\begin{proposition}\label{propexample}
There is a compact set $K \subset \mathbb{R}$ and jets $F$, $G$, and $H$ of order $m$ on $K$ 
so that
\begin{enumerate}
\item $F$, $G$, and $H$ are Whitney fields of class $C^m$ on $K$,
\item for every $1 \leq k \leq m$ and $t \in K$ we have
$$
H^{k}(t) = \mathcal{P}^k \left(F^0(t),G^0(t),F^1(t),G^1(t),\dots,F^{k}(t),G^{k}(t)\right),
$$
\end{enumerate}
but there is no $C^m$ horizontal curve $(f,g,h): \mathbb{R} \to \mathbb{H}^1$ extending the triple $(F,G,H)$.
\end{proposition}
\begin{proof}
As mentioned above, we will proceed as in the proof of \cite[Proposition 1.3]{Zim18}.
Define the compact set $K \subset \mathbb{R}$ as follows:
$$
K := \bigcup_{n=0}^{\infty} [c_n,d_n] \cup \{1\} 
\quad
\text{where }
[c_n,d_n] := \left[ 1- 2^{-n},1-\tfrac34 2^{-n} \right]
\text{ for every } n \in \mathbb{N}.
$$
For $0 \leq k \leq m$, define $F^k(t)=0$ and $G^k(t)=0$ for every $t \in K$.
Also, define $H(t) = 3^{-mn}$ if $t \in [c_n,d_n]$ and $H(1) = 0$,
and set $H^k(t) = 0$ for $1 \leq k \leq m$ and every $t \in K$.
The jets of $F$ and $G$ are trivially Whitney fields,
and (2) is trivially satisfied.

We will now show that $H$ is a Whitney field.
For $1 \leq k \leq m$, the remainders $(R_a^m H)^k$ are constantly 0 on $K$.
Thus we need only show that 
$$
\frac{(R_a^m H)^0(b)}{|b-a|^m} = \frac{H(b) - H(a)}{|b-a|^m} \to 0
$$ 
uniformly on $K$ as $(b-a) \to 0$.
Fix $\varepsilon > 0$ and $n \in \mathbb{N}$ with $4^m(2/3)^{mn} < \varepsilon$.
Choose $a,b \in K$ with $|b-a| < 2^{-(n+2)}$.
If $a$ and $b$ lie in the same interval $[c_k,d_k]$, then $(R_a^m H)^0(b) = 0$.
If $a$ and $b$ lie in different intervals $[c_k,d_k]$ and $[c_\ell,d_\ell]$
(say $\ell > k$)
then, as in the proof of \cite[Proposition 1.3]{Zim18}, we see that $k \geq n$.
Therefore
$$
\frac{H(b) - H(a)}{|b-a|^m} 
\leq 
\frac{3^{-mk} - 3^{-m\ell}}{(c_\ell-d_k)^m}
\leq 
\frac{3^{-mk}}{(c_{k+1}-d_k)^m}
=
4^m\left( \tfrac{2}{3} \right)^{mk}
\leq 
4^m\left( \tfrac{2}{3} \right)^{mn}
<\varepsilon.
$$
If either $a$ or $b$ is equal to 1, a similar argument holds.
Hence $H$ is a Whitney field.

Suppose now that $(f,g,h): \mathbb{R} \to \mathbb{H}^1$ is a $C^m$ curve extending $(F,G,H)$.
(According to the classical Whitney extension theorem, such a curve is guaranteed to exist.)
Suppose also that $(f,g,h)$ is horizontal.
Then, according to Proposition~\ref{nec},
we must have $A(a,b)/V(a,b) \to 0$ uniformly as $(b-a) \to 0$ for $a,b \in K$.
However,
$$
A(a,b) = H(b)-H(a)
\quad
\text{and}
\quad
V(a,b) = (b-a)^{2m}
$$
for any $a,b \in K$.
Therefore, we have $|c_{n+1}-d_n| = 2^{-(n+2)} \to 0$, but
$$
\frac{A(c_{n+1},d_n)}{V(c_{n+1},d_n)} 
= \frac{3^{-mn} - 3^{-m(n+1)}}{4^{-m(n+2)}}
= \frac{(3^m-1)16^m}{3^m} \left( \frac{4}{3} \right)^{mn}
\to \infty
$$
as $n \to \infty$.
This contradicts Proposition~\ref{nec}.
Thus there is no $C^m$ horizontal curve extending $(F,G,H)$.
\end{proof}

\section{Necessity of the criteria for a $C^m$ horizontal extension}\label{sectionnecessary}

In this section we show that the conditions in Theorem \ref{iff} 
are necessary for existence of a $C^m$ horizontal extension. Recall the polynomials $\mathcal{P}^k$ from \eqref{HigherHoriz}.

\begin{proposition}\label{nec}
Suppose $(f,g,h)\colon \mathbb{R}\to \mathbb{H}^{1}$ is a $C^m$ horizontal curve and $K\subset \mathbb{R}$ is a compact set. 
Let $F=J^m(f)|_K$, $G=J^m(g)|_K$, and $H = J^m(h)|_K$ be the jets of order $m$ obtained by restricting $f,g,h$ and their derivatives to $K$. Then 
\begin{enumerate}
\item $F$ $G$, and $H$ are Whitney fields of class $C^m$ on $K$, 
\item for all $t\in K$ and $1 \leq k \leq m$ we have
\begin{equation}
\label{Horiznec}
H^{k}(t) = \mathcal{P}^k \left(F^0(t),G^0(t),F^1(t),G^1(t),\dots,F^{k}(t),G^{k}(t)\right),
\end{equation}
\item $A(a,b)/V(a,b)\to 0$ uniformly as $(b-a) \searrow 0$ with $a,b\in K$.
\end{enumerate}
\end{proposition}

We use the remainder of this section to prove Proposition \ref{nec}.

\begin{proof}
Suppose $f, g, h, F, G, H, K$ are as in the statement of Proposition \ref{nec}. 
Without loss of generality, we may assume that $K=[A,B]$ is a closed interval.
Indeed, if (1), (2), and (3) hold on the interval $[A,B]$, then they also hold on any compact subset.
Taylor's theorem asserts that $F$, $G$, and $H$ must be Whitney fields of class $C^m$ on $K$. Also \eqref{HigherHoriz} gives
\[D^kh(t) = \mathcal{P}^k \left(f(t),g(t),f'(t),g'(t),\dots,D^kf(t),D^kg(t)\right)\]
for $1 \leq k \leq m$ and for all $t\in \mathbb{R}$. This proves Proposition \ref{nec} (1) and (2). It remains to prove (3).

Fix $\varepsilon>0$. There exists $\delta>0$ such that if $[a,b] \subset K$ and $(b-a)<\delta$ then:
\begin{enumerate}[(i)]
\item $|D^if-D^if(a)|\leq \varepsilon$ and $|D^ig-D^ig(a)|\leq \varepsilon$ on $[a,b]$ for $0\leq i\leq m$.
\item $|f-T_a^mf|\leq \varepsilon (b-a)^m$ and $|g-T_a^mg|\leq \varepsilon (b-a)^m$ on $[a,b]$.
\item $|f'-(T_a^mf)'|\leq \varepsilon (b-a)^{m-1}$ and $|g'-(T_a^mg)'|\leq \varepsilon (b-a)^{m-1}$ on $[a,b]$.
\end{enumerate}
Let $a\in K$ and let $Tf=T_a^mf$ and $Tg=T_a^mg$ be the Taylor polynomials of $f$ and $g$ of order $m$ at $a$.
Fix $b\in K$ with $0<b-a<\delta$.

\vspace{.2cm}

\textbf{Temporarily assume $f(a)=g(a)=h(a)=0$.} In this case $A(a,b)$ takes the simpler form.
\[ A(a,b) = h(b)-h(a)-2 \int_{a}^{b} ((Tf)'Tg-Tf(Tg)'). \]
Since $(f,g,h)$ is a horizontal curve, we have
\[h(b)-h(a)=2\int_{a}^{b} (f'g-fg').\]
Hence we can estimate $|A(a,b)|$ as follows
\begin{align}\label{kareaest}
&\left| h(b)-h(a)-2\int_{a}^{b} ((Tf)'Tg-Tf(Tg)') \right|\\
&\qquad \leq 2\int_{a}^{b}|f'g-(Tf)'Tg| + 2\int_{a}^{b}|fg'-Tf(Tg)'|.\nonumber
\end{align}
We will only show how to bound the first term above, since the second is the same with $f$ and $g$ interchanged. Notice
\[f'g-(Tf)'(Tg) = (f'-(Tf)')(g-Tg) + (f'-(Tf)')(Tg) + (g-Tg)(Tf)'.\]
Hence
\begin{align*}
\int_{a}^{b}|f'g-(Tf)'Tg|&\leq \int_{a}^{b}|f'-(Tf)'||g-Tg| + |f'-(Tf)'||Tg| + |g-Tg||(Tf')|\\
&\leq \varepsilon^{2}(b-a)^{2m} + \varepsilon(b-a)^{m-1}\int_{a}^{b}|Tg| + \varepsilon(b-a)^{m}\int_{a}^{b}|(Tf)'|.
\end{align*}
Using a similar estimate for the second term gives the following estimate of \eqref{kareaest}
\begin{align*}
& \left| h(b)-h(a)-2 \int_{a}^{b} ((Tf)'Tg-Tf(Tg)') \right|\\
&\leq 4\varepsilon^{2}(b-a)^{2m} +2\varepsilon (b-a)^{m-1}\int_{a}^{b}|Tf|+|Tg| + 2\varepsilon (b-a)^{m}\int_{a}^{b}|(Tf)'|+|(Tg)'|.
\end{align*}
Since $Tf(a)=f(a)=0$, we have $|Tf(x)|\leq M(b-a)$ on $[a,b]$ where $M=\max_{[a,b]}|(Tf)'|$. Combining this with Corollary \ref{intmax} applied to the polynomial $(Tf)'$ gives
\[\int_{a}^{b}|Tf|\leq M(b-a)^{2}\leq 8m^{2}(b-a)\int_{a}^{b}|(Tf)'|.\]
Similarly, we obtain the same inequality for $g$
\[\int_{a}^{b}|Tg|\leq M(b-a)^{2}\leq 8m^{2}(b-a)\int_{a}^{b}|(Tg)'|.\]
Hence
\begin{align*}
& \left| h(b)-h(a)-2 \int_{a}^{b} ((Tf)'Tg-Tf(Tg)') \right|\\
&\qquad \leq 4\varepsilon^{2}(b-a)^{2m} + (2+16m^2)\varepsilon (b-a)^{m}\int_{a}^{b}|(Tf)'|+|(Tg)'|.
\end{align*}

\vspace{.2cm}

\textbf{General case without assuming $f(a)=g(a)=h(a)=0$}. We begin by considering the curve $(u,v,w):=(f(a),g(a),h(a))^{-1}(f,g,h)$. The definition of the group operation gives
\begin{equation}\label{horiztranslate}
u=f-f(a), \quad v=g-g(a),
\end{equation}
\begin{equation}\label{verttranslate}
w=h-h(a)+2f(a)g-2g(a)f.
\end{equation}
Also $D^iu=D^if$ and $D^iv=D^ig$ for $1\leq i\leq m$. Clearly $u(a)=v(a)=w(a)=0$ and $(u,v,w)$ is a $C^m$ horizontal curve. 
As a consequence of this, we obtain the following analogues of the earlier estimates on $f$ and $g$ for $0\leq i \leq m$ on the interval $[a,b]$:
\begin{enumerate}[(i)]
\item $|D^iu-D^iu(a)|=|D^if-D^if(a)|\leq \varepsilon$, $|D^iv-D^iv(a)|=|D^ig-D^ig(a)|\leq \varepsilon$.
\item $|u-Tu|=|f-Tf|\leq \varepsilon (b-a)^m$ and $|v-Tv|=|g-Tg|\leq \varepsilon (b-a)^m$.
\item $|u'-(Tu)'|=|f'-(Tf)'|\leq \varepsilon (b-a)^{m-1}$ and $|v'-(Tv)'|=|g'-(Tg)'|\leq \varepsilon (b-a)^{m-1}$.
\end{enumerate}

Hence we may follow the proof of the previous case to obtain the estimate

\begin{align*}
& \left| w(b)-w(a)-2 \int_{a}^{b} ((Tu)'Tv-Tu(Tv)') \right|\\
&\qquad \leq 4\varepsilon^{2}(b-a)^{2m} + (2+16m^2)\varepsilon (b-a)^{m}\int_{a}^{b}|(Tu)'|+|(Tv)'|.
\end{align*}

Easy calculations yield
\[w(b)-w(a)=h(b)-h(a)+2f(a)g(b)-2g(a)f(b)\]
and
\[\int_{a}^{b} ((Tu)'Tv-Tu(Tv)')  = \int_{a}^{b} ((Tf)'(Tg)-(Tf)(Tg)') -g(a)Tf(b)+f(a)Tg(b).\]
Hence
\begin{align*}
&w(b)-w(a)-2 \int_{a}^{b} ((Tu)'Tv-Tu(Tv)')\\
&\qquad = h(b)-h(a)-2\int_{a}^{b} (Tf)'(Tg)-(Tf)(Tg)' \\
&\qquad \qquad +2f(a)(g(b)-Tg(b))- 2g(a)(f(b)-Tf(b)).    \\
\end{align*}

We deduce that the absolute value of
\begin{align*}
& h(b)-h(a)-2\int_{a}^{b} (Tf)'(Tg)-(Tf)(Tg)' \\
& \qquad +2f(a)(g(b)-Tg(b))- 2g(a)(f(b)-Tf(b))
\end{align*}
is less than or equal to 
\[4\varepsilon^{2}(b-a)^{2m} + (2+16m^2)\varepsilon (b-a)^{m}\int_{a}^{b}|(Tf)'|+|(Tg)'|.\]

That is, we have shown $|A(a,b)|\leq (4\varepsilon^{2} + (2+16m^2)\varepsilon)V(a,b)$ for any $a,b\in K$ with $0<b-a<\delta$. This shows 
\[A(a,b)/V(a,b)\to 0 \mbox{ uniformly as }(b-a)\searrow 0 \mbox{ with }a,b\in K\]
which concludes the proof of Proposition \ref{nec}.
\end{proof}

\section{Sufficiency of the criteria for a $C^m$ horizontal extension}\label{sectionsufficient}

In this section we show that the conditions in Theorem~\ref{iff} are sufficient to guarantee the existence of a $C^m$ horizontal extension. Recall the polynomials $\mathcal{P}^k$ from \eqref{HigherHoriz} obtained by differentiating the definition of a horizontal curve. Given jets $F, G, H$ of order $m$ on a compact set $K\subset \mathbb{R}$ and $a,b\in K$, recall the area discrepancy $A(a,b)$ and velocity $V(a,b)$ from \eqref{Aab} and \eqref{Vab} given by:
\begin{align*}
 A(a,b)&=H(b)-H(a)-2\int_{a}^{b}((T_{a}F)'(T_{a}G)-(T_{a}G)'(T_{a}F))\\
\nonumber &\qquad +2F(a)(G(b)-T_{a}G(b))-2G(a)(F(b)-T_{a}F(b)),
\end{align*}
and
\begin{equation*}
V(a,b)= (b-a)^{2m} + (b-a)^{m} \int_{a}^{b}|(T_{a}F)'|+|(T_{a}G)'|.
\end{equation*}

\begin{theorem}\label{newthm}
Let $K\subset \bbR$ be compact and $F,G,H$ be jets of order $m$ on $K$. Assume
\begin{enumerate}
\item $F$, $G$, and $H$ are Whitney fields of class $C^m$ on $K$,
\item for every $1 \leq k \leq m$ and $t\in K$ we have
$$
H^{k}(t) = \mathcal{P}^k \left(F^0(t),G^0(t),F^1(t),G^1(t),\dots,F^{k}(t),G^{k}(t)\right),
$$
\item $A(a,b)/V(a,b)\to 0$ uniformly as $(b-a) \searrow 0$ with $a,b\in K$.
\end{enumerate}
Then the triple $(F,G,H)$ extends to a $C^m$ horizontal curve $(f,g,h)\colon \bbR \to \bbH^{1}$.
\end{theorem}

We use the remainder of this section to prove Theorem \ref{newthm}. 

\begin{proof}
Suppose $K,F,G,H$ are as in the statement of the theorem and satisfy the assumptions stated. Let $I=[\min K, \max K]$ and notice that it suffices to find a $C^m$ horizontal extension $(f,g,h)\colon I\to \bbH$. Here, derivatives and continuity at the endpoints are defined, as usual, using one-sided limits. We may write $I\setminus K=\cup_{i=1}^{\infty}(a_i,b_i)$ for disjoint open intervals $(a_i,b_i)$ with $a_i, b_i \in K$.

Using the classical Whitney extension theorem (Theorem~\ref{classicalWhitney}), we can choose $f,g\colon I\to \bbR$ of class $C^m$ such that $D^kf(x)=F^{k}(x)$ and $D^kg(x)=G^{k}(x)$ for every $x\in K$ and $0\leq k\leq m$. We also choose $f$ and $g$ to be $C^{\infty}$ in $I\setminus K$. Note that, while the classical Whitney extension theorem gives extension to an open set containing $K$, to extend to $I$ we simply extend to an open set containing $I$ then restrict to $I$. 
Recall that $D^{k}(T_{a}^{m}F)(x)$ takes the form
\[ D^{k}(T_{a}^{m}F)(x)=\sum_{\ell=0}^{m-k}\frac{F^{k+\ell}(a)}{\ell!}(x-a)^{\ell}\]
and a similar expression holds for $D^{k}(T_{a}^{m}G)(x)$. 
Using \eqref{modulus1}, we may assume there exists a modulus of continuity $\alpha$ so that, for all $a, x\in K$ and $0\leq k\leq m$,
\begin{equation}\label{modulus1F}
|F^{k}(x)-D^{k}(T_{a}^{m}F)(x) | \leq \alpha(|x-a|)|x-a|^{m-k},
\end{equation}
\begin{equation}\label{modulus1G}
|G^{k}(x)-D^{k}(T_{a}^{m}G)(x) | \leq \alpha(|x-a|)|x-a|^{m-k}.
\end{equation}
Using \eqref{modulus2}, we can ensure that for $a\in K$, $x\in I$ and $0\leq k\leq m$:
\begin{equation}\label{modulus2F}
|D^kf(x)-D^{k}(T_{a}^{m}F)(x)| \leq C\alpha(|x-a|)|x-a|^{m-k},
\end{equation}
\begin{equation}\label{modulus2G}
|D^kg(x)-D^{k}(T_{a}^{m}G)(x)| \leq C\alpha(|x-a|)|x-a|^{m-k}
\end{equation}
for some constant $C>0$.
Hence, by scaling the value of $\alpha$ by a constant depending on $F, G, K, m$ (but still maintaining $\alpha(t)\to 0$ as $t\searrow 0$), we can assume that for every $0\leq k\leq m$:
\begin{equation}\label{modcty}
|D^kf(x)-D^kf(a)|\leq \alpha(|x-a|) \mbox{ and } |D^kg(x)-D^kg(a)|\leq \alpha(|x-a|).
\end{equation}
Finally, using the hypothesis $A(a,b)/V(a,b)\to 0$ uniformly, we choose $\alpha$ possibly larger (but still a modulus of continuity) so that
\begin{equation}\label{quotientcvg}
A(a,b)\leq \alpha(b-a)V(a,b) \mbox{ for }a,b\in K \mbox{ with }a<b.
\end{equation}

\begin{proposition}\label{perturb}
There exists a modulus of continuity $\beta \geq \alpha$ (independent of $i$) for which the following holds: 
for each interval $[a_i,b_i]$, there exist $C^{\infty}$ functions $\phi, \psi\colon [a_i,b_i]\to \bbR$ such that
\begin{enumerate}
\item $D^k\phi(a_i)=D^k\phi(b_i)=D^k\psi(a_i)=D^k\psi(b_i)=0$ for $0\leq k\leq m$.
\item $\max\{|D^k\phi|, |D^k\psi|\} \leq \beta(b_i-a_i)$ for $0\leq k\leq m$ on $[a_i,b_i]$.
\item $H(b_i)-H(a_i)=2\int_{a_i}^{b_i}(f+\phi)'(g+\psi)-(g+\psi)'(f+\phi)$.
\end{enumerate}
\end{proposition}

\begin{proof}

This proof will require several lemmas. Fix an interval $[a,b]$ of the form $[a_i,b_i]$ for some $i$.
 Since the interval is fixed, we will write $A, V, \alpha$ instead of $A(a,b), V(a,b), \alpha(b-a)$ respectively. 

\begin{claim}
It suffices to consider the case $F(a)=G(a)=H(a)=0$.
\end{claim}

\begin{proof}
Indeed, suppose this case has been established and $f, g, F, G, H$ are chosen without restriction satisfying \eqref{modulus1F}--\eqref{quotientcvg}. Define $u=f-f(a)$, $v=g-g(a)$ and
define jets $U,V,W$ so that 
\[(U,V,W)=(F(a),G(a),H(a))^{-1}(F,G,H)\]
and $U^k=F^k$ and $V^k=G^k$ for $1 \leq k \leq m$.
The equations for $U,V,W$ are calculated from the group law as in \eqref{horiztranslate} and \eqref{verttranslate}. It is easy to verify that the analogues of \eqref{modulus1F}--\eqref{modcty} hold for $u,v,U,V,W$ since the group law is simply Euclidean addition in the two horizontal directions. 
For example, \eqref{modulus1F} becomes 
\[|U^{k}(x)-D^{k}(T_{a}^{m}U)(x) | \leq \alpha(|x-a|)|x-a|^{m-k}.\]
Now, a simple calculation yields
\begin{align*}
 A(a,b)&:=H(b)-H(a)-2\int_{a}^{b}((T_{a}F)'(T_{a}G)-(T_{a}G)'(T_{a}F))\\
 &\qquad +2F(a)(G(b)-T_{a}G(b))-2G(a)(F(b)-T_{a}F(b)).\\
&= W(b)-2\int_{a}^{b}(T_aU)'(T_aV)-(T_aV)'(T_aU)
\end{align*}
which is the analogue of $A(a,b)$ for $(U,V,W)$. Of course,
\begin{align*}
V(a,b)&= (b-a)^{2m} + (b-a)^{m} \int_{a}^{b}|(T_{a}F)'|+|(T_{a}G)'|\\
&=(b-a)^{2m} + (b-a)^{m} \int_{a}^{b}|(T_{a}U)'|+|(T_{a}V)'|\\
\end{align*}
which is the analogue of $V(a,b)$ for $(U,V,W)$. Hence we obtain also the analogue of \eqref{quotientcvg} for $U,V,W$. If we assume that the proposition has been proven for $u,v,U,V,W$ satisfying the initial condition $U(a)=V(a)=W(a)=0$, then this gives a modulus of continuity $\beta$ and $C^{\infty}$ functions $\phi, \psi\colon [a,b]\to \bbR$ such that
\begin{enumerate}
\item $D^k\phi(a)=D^k\phi(b)=D^k\psi(a)=D^k\psi(b)=0$ for $0\leq k\leq m$.
\item $\max \{|D^k\phi|, |D^k\psi| \}\leq \beta(b-a)$ for $0\leq k\leq m$ on $[a,b]$.
\item $W(b)-W(a)=2\int_{a}^{b}(u+\phi)'(v+\psi)-(v+\psi)'(u+\phi)$.
\end{enumerate}
Simple calculations yield
\[W(b)-W(a)=W(b)=H(b)-H(a)+2F(a)G(b)-2G(a)F(b)\]
and
\begin{align*}
2\int_{a}^{b}(u+\phi)'(v+\psi)-(v+\psi)'(u+\phi)&= 2\int_{a}^{b}((f+\phi)'(g+\psi)-(g+\psi)'(f+\phi))\\
&\qquad  +2f(a)g(b)-2g(a)f(b).
\end{align*}
Hence (3) is transformed into
\[H(b)-H(a)=2\int_{a}^{b}(f+\phi)'(g+\psi)-(g+\psi)'(f+\phi).\]
This is the desired statement for the general curve, 
and the claim is proven.
\end{proof}

Hence we can assume $F(a)=G(a)=H(a)=0$. In this case we have
\[A=H(b)-H(a)-2\int_{a}^{b}((TF)'(TG)-(TG)'(TF))\]
and
\[V= (b-a)^{2m} + (b-a)^{m} \int_{a}^{b}|(TF)'|+|(TG)'|. \]
Define
\[ \mathcal{A}:=H(b)-H(a)-2\int_{a}^{b}(f'g-g'f).\]
Notice Proposition \ref{perturb}(3) can be rewritten as:
\begin{equation}\label{pertforAcal}
2\int_{a}^{b}(f'\psi-\psi'f)+(\phi'g-g'\phi)+(\phi'\psi-\psi'\phi)=\mathcal{A}.
\end{equation}

An argument analogous to the one in Section \ref{sectionnecessary} with $\varepsilon$ replaced by $\alpha$ yields
\[ \left| \int_{a}^{b}(f'g-g'f)-\int_{a}^{b} ((Tf)'Tg-Tf(Tg)') \right| \leq (2\alpha^{2}+(1+8m^2)\alpha)V. 
\]
Combining this with \eqref{quotientcvg} shows that
\begin{equation}\label{estforperturb}
|\mathcal{A}| 
= \left| H(b)-H(a)-2\int_{a}^{b} (f'g-g'f) \right| 
\leq C_1 \hat{\alpha} V.
\end{equation}
Here, $\hat{\alpha} = \alpha^2 + \alpha$, and $C_1 \geq 1$ is a fixed constant depending on $m$ which we will refer to later. 
Intuitively, \eqref{estforperturb} implies that $f$ and $g$ are no worse than $TF=Tf$ and $TG=Tg$ for the purpose of lifting to give the correct height. However, they have the advantage of being $C^m$ maps defined on the interval $[a,b]$ which satisfy the correct boundary conditions.
Next, 
\[\int_{a}^{b}(f'\psi-\psi'f)=2\int_{a}^{b}f' \psi\]
for any $C^{\infty}$ function $\psi$ which vanishes at $a$ and $b$,
and a similar equation holds for the other terms in \eqref{pertforAcal}. 
Hence constructing $\phi$ and $\psi$ which satisfy Proposition \ref{perturb}(3) is equivalent to solving
\begin{equation}
\label{goalint}
4\int_{a}^{b}(\psi f'+\phi g' + \psi \phi')=\mathcal{A},
\end{equation}
where $\mathcal{A}$ satisfies $|\mathcal{A}|\leq C_1 \hat{\alpha} V$. We now show how to do this subject to the constraints Proposition~\ref{perturb}(1) and (2). 

\begin{remark}
\label{alphabound}
Suppose $C_0 > 0$ is a constant 
depending only on $m$ and $\text{diam}(K)$
(the exact value of which will be established at the beginning of Lemma~\ref{fbig}).
Note that, once this constant has been chosen, 
we may assume for the rest of the proof of the proposition that 
$\hat{\alpha} < 1/C_0$. 
Indeed, since $\hat{\alpha}$ is a modulus of continuity and $I$ is bounded, 
there are only finitely many intervals $[a_i,b_i]$ satisfying $\hat{\alpha}(b_i-a_i) \geq 1/C_0$.
Thus, if we are currently considering such an interval, 
we are free to choose $\psi$ and $\phi$ to be any $C^\infty$ functions which satisfy \eqref{goalint} and Proposition~\ref{perturb}(1),
and we may assign $\beta(b_i - a_i)$ to be the maximum over all $|D^i \phi(x)|$ and $|D^i \psi(x)|$ for $0 \leq i \leq m$ and $x \in [a,b]$
(and over all $[a_j,b_j]$ with $b_j-a_j = b_i-a_i$)
to ensure that Proposition~\ref{perturb}(2) holds.
(We may also choose $\beta$ in such a way that it is still an increasing function.)
\end{remark}

We divide the constructions of $\phi$ and $\psi$ into two cases.
In one case, $(Tf)'$ (or $(Tg)'$) is large enough on average
to allow us to create a controlled perturbation of $g$ (or $f$) that encloses the prescribed area (Proposition~\ref{perturb}(3)).
In the other case, both $(Tf)'$ and $(Tg)'$ are small on average, and so $A$ is small as well.
Thus, both $f$ and $g$ may be perturbed slightly to satisfy Proposition~\ref{perturb}(3).
For convenience, we recall
\[V= (b-a)^{2m} + (b-a)^{m} \int_{a}^{b}|(Tf)'|+|(Tg)'|. \]

\begin{lemma}\label{fbig}
Suppose
\begin{equation}
\label{fbigassumption}
\int_{a}^{b} |(Tf)'| \geq \max \left( \int_{a}^{b} |(Tg)'|,\ (b-a)^{m} \right).
\end{equation}
Then there exists a $C^{\infty}$ map $\psi$ on $[a,b]$ satisfying
\begin{enumerate}
\item $D^i\psi(a)=D^i\psi(b)=0$ for $0\leq i\leq m$,
\item $|D^i\psi(x)|\leq (C_1C_0)\hat{\alpha}$ on $[a,b]$,
\item $4 \int_{a}^{b} \psi f' =\mathcal{A}$.
\end{enumerate}
\end{lemma}
Hence,
if \eqref{fbigassumption} holds,
we may choose $\phi \equiv 0$ on $[a,b]$.
\begin{proof}
In this case we have 
\begin{equation}
\label{Abound1}
|\mathcal{A}|\leq 3C_1\hat{\alpha}(b-a)^{m}\int_{a}^{b}|(Tf)'|.
\end{equation}
We begin by applying Lemma \ref{Pbig} to the polynomial $(Tf)'$. 
This gives a closed subinterval $I\subset [a,b]$ of length at least $(b-a)/4m^{2}$ such that $|(Tf)'|\geq M/2$ in $I$, where $M=\max_{[a,b]}|(Tf)'|$. 
In partiular, $(Tf)' \neq 0$ on $I$.
By rescaling, translating, and dilating a bump function, we can choose a $C^{\infty}$ map $\eta$ on $[a,b]$ 
and a constant $C_0 \geq 1$ depending only on $m$ and $\text{diam}(K)$
such that
\begin{enumerate}[(a)]
\item $D^i\eta=0$ outside $I$ for $0\leq i\leq m$,
\item $|D^i\eta(x)| \leq C_0 \hat{\alpha}$ on $[a,b]$,
\item $|\eta| \geq 48m^2 \hat{\alpha} \cdot (b-a)^{m}$ on the interval consisting of the middle third of $I$ (which has length at least $(b-a)/12m^{2}$) so that the sign of $\eta$ is the same as the sign of $(Tf)'$ on $I$.
\end{enumerate}
Now define $\psi$ on $[a,b]$ by scaling $\eta$ by a constant:
$$
\psi = \left( \frac{\mathcal{A}}{4 \int_a^b \eta f'} \right) \cdot \eta.
$$
In particular, this gives $4 \int_{a}^{b} \psi f' =\mathcal{A}$ which is property (3).
Clearly, $\psi$ satisfies property (1).
It remains to show that $\psi$ satisfies property (2).
To prove this, we must bound $\left| \int_a^b \eta f' \right|$ from below.
We begin with
\begin{align*}
\left| \int_{a}^{b} \eta (Tf)' \right|
=\int_I \eta (Tf)'
\geq \frac{b-a}{12m^2} 48m^2 \hat{\alpha} (b-a)^{m} \frac{M}{2}
&= 2 \hat{\alpha} (b-a)^{m+1}  M \\
&\geq 2 \hat{\alpha} (b-a)^{m} \int_{a}^{b}|(Tf)'|.
\end{align*}
Using $|f'-(Tf)'|\leq \hat{\alpha} (b-a)^{m-1}$ (since $\alpha \leq \hat{\alpha}$) 
and $|\eta|\leq C_0 \hat{\alpha}(b-a)^{m}$ gives
\[\int_{a}^{b}| \eta f' -\eta(Tf)'|
\leq C_0 \hat{\alpha}^2 (b-a)^{2m}
\leq C_0 \hat{\alpha}^2 (b-a)^{m} \int_{a}^{b}|(Tf)'|. \]
This gives
\begin{equation}
\label{ebound}
\left| \int_{a}^{b}\eta f' \right| 
\geq \hat{\alpha} (b-a)^{m} \int_{a}^{b}|(Tf)'| \left( 2 - C_0 \hat{\alpha} \right)
>\hat{\alpha} (b-a)^{m} \int_{a}^{b}|(Tf)'|
\end{equation}
because of the assumption $\hat{\alpha} < 1/C_0$ from Remark~\ref{alphabound}.
We are now ready to finish the proof of (2).
From \eqref{Abound1} and \eqref{ebound}, we have
$$
|D^i\psi| 
= \frac{|\mathcal{A}|}{4 \left| \int_a^b \eta f' \right|} |D^i \eta|
\leq (C_1C_0)\hat{\alpha}.
$$
\end{proof}

Clearly Lemma \ref{fbig} has a direct analogue involving $\phi$ if 
\begin{equation}
\label{gbigassumption}
\int_{a}^{b} |(Tg)'| \geq \max \left( \int_{a}^{b} |(Tf)'|,\ (b-a)^{m} \right),
\end{equation}
in which case the conclusion is $4 \int_{a}^{b} \phi g' =\mathcal{A}$. 
In this case, we choose $\psi \equiv 0$ on $[a,b]$.
It remains to construct the functions $\phi$ and $\psi$ when \eqref{fbigassumption} and \eqref{gbigassumption} both fail.

\begin{lemma}\label{fgsmall}
Suppose
\[(b-a)^{m} > \max \left( \int_{a}^{b} |(Tf)'|,\ \int_{a}^{b} |(Tg)'| \right).\]
Then there exist $C^{\infty}$ maps $\phi$ and $\psi$ 
and a constant $C_2 > 0$ depending only on $m$ and $\text{diam}(K)$
such that
\begin{enumerate}
\item $D^i\psi(a)=D^i\psi(b)=D^i\phi(a)=D^i\phi(b)=0$ for $0\leq i\leq m$,
\item $\max \{|D^i\psi(x)|, |D^i\phi(x)| \} \leq 18C_2 \sqrt{C_1\hat{\alpha}}$ on $[a,b]$ for $0\leq i\leq m$,
\item $4 \int_{a}^{b} (\psi f' + \phi g' + \psi \phi') =\mathcal{A}$.
\end{enumerate}
\end{lemma}

\begin{proof}
Notice the hypotheses imply
\begin{equation}\label{Abound}
 |\mathcal{A}| \leq C_{1} \hat{\alpha} V \leq 3C_{1} \hat{\alpha} (b-a)^{2m}.
\end{equation}
For simplicity, write $B := 3\sqrt{C_1\hat{\alpha}}$.
Set $C_{2} > 0$ to be a constant (depending only on $m$ and $\text{diam}(K)$) so that 
there exist $C^{\infty}$ functions $\xi$ and $\eta$ on $[a,b]$ 
(which are each a dilated, scaled, translated bump function)
so that 
\begin{enumerate}[(a)]
\item $D^i \xi(a) = D^i \xi(b) = 0$ for $0 \leq i \leq m$,
\item $|D^i\xi| \leq C_2 B$ on $[a,b]$ for $0 \leq i \leq m$,
\item $\xi' \geq B(b-a)^{m-1}$ on a subinterval $I$ of $[a,b]$ with length $(b-a)/3$.
\end{enumerate}
and
\begin{enumerate}[(a)]
\item $D^i \eta$ vanishes outside of $I$ for $0 \leq i \leq m$,
\item $|D^i\eta| \leq C_2 B$ on $[a,b]$ for $0 \leq i \leq m$,
\item $\eta\geq B(b-a)^{m}$ on the middle third of $I$ (which has length $(b-a)/9$).
\end{enumerate}

\textbf{Case 1:}
$\left| \int_{a}^{b} \eta f' \right| \geq  |\mathcal{A}|/24$.

Set $\phi \equiv 0$ on $[a,b]$ and $\psi = \eta\cdot ( \mathcal{A}/(4\int_a^b \eta f'))$. 
Then (1) and (3) are clearly satisfied, 
and we have as before
$$
|D^i\psi| 
= \frac{|\mathcal{A}|}{4 \left| \int_a^b \eta f' \right|} |D^i \eta|
\leq 6C_2B
\quad
\text{ on }
[a,b].
$$

\textbf{Case 2:}
$\left| \int_{a}^{b} \xi g' \right| \geq  |\mathcal{A}|/24$.
This is identical to the previous case when we choose $\psi \equiv 0$ on $[a,b]$ and $\phi = \xi \cdot ( \mathcal{A}/(4\int_a^b \xi g'))$.

\textbf{Case 3:}
$\max \left\{ \left| \int_{a}^{b} \eta f' \right|, \left|\int_{a}^{b} \xi g'  \right| \right\} <  |\mathcal{A}|/24$.

We first have
$$
\int_{a}^{b} \eta \xi' 
= \int_I \eta \xi'
\geq \tfrac19 B^{2}(b-a)^{2m} = C_{1} \hat{\alpha} (b-a)^{2m} \geq \frac{|\mathcal{A}|}{3},
$$
and hence
$ 4\int_{a}^{b} (\eta f' + \xi g' + \eta \xi') > 
|\mathcal{A}|$.
Now consider $\mathcal{F}\colon \bbR\to \bbR$ defined by 
\[\mathcal{F}(\lambda)=4\int_{a}^{b} (\lambda \eta) f' + \xi g' + (\lambda \eta) \xi'. \]
Clearly $\mathcal{F}$ is a continuous map with $\mathcal{F}(0)=4\int_{a}^{b}\xi g' < |\mathcal{A}|/6$ and $\mathcal{F}(1)>|\mathcal{A}|$. Hence, by the intermediate value theorem, there exists $\lambda \in (0,1)$ such that $\mathcal{F}(\lambda)=|\mathcal{A}|$. 
In other words, if we set $\phi = \pm \xi$ and $\psi = \pm \lambda \eta$
(with the appropriate choice of sign),
we have found mappings which satisfy the lemma.
\end{proof}

Setting $\beta(b-a) := \max \{(C_1C_0)\hat{\alpha},18C_2 \sqrt{C_1\hat{\alpha}} \}$ 
completes the proof of Proposition~\ref{perturb}.
\end{proof}

We are now ready to build the $C^m$ horizontal extension of $(F,G,H)$.

\begin{lemma}
\label{buildcurve}
For all $i\geq 1$, there is a horizontal curve $(\mathcal{F}_{i}, \mathcal{G}_{i}, \mathcal{H}_{i})\colon [a_{i},b_{i}]\to \bbH^{1}$ of class $C^m$ which satisfies
\begin{enumerate}
\item for $0\leq k\leq m$,
$$
\begin{array}{ccc}
D^k \mathcal{F}_{i}(a_{i})=F^{k}(a_{i}), & D^k \mathcal{G}_{i}(a_{i})=G^{k}(a_{i}), & D^k \mathcal{H}_{i}(a_{i})=H^{k}(a_{i}) \\
D^k \mathcal{F}_{i}(b_{i})=F^{k}(b_{i}), & D^k \mathcal{G}_{i}(b_{i})=G^{k}(b_{i}), & D^k \mathcal{H}_{i}(b_{i})=H^{k}(b_{i}) 
\end{array}
$$
\item $|D^k \mathcal{F}_{i}(x)-F^{k}(a_{i})|\leq 2\beta(b_{i}-a_i)$ 
and 
$|D^k \mathcal{G}_{i}(x)-G^{k}(a_{i})|\leq 2\beta(b_{i}-a_i)$
for $0\leq k\leq m$ and $x\in [a_{i},b_{i}]$. 
\end{enumerate}
\end{lemma}

\begin{proof}
Fix $i \in \mathbb{N}$.
Set $\mathcal{F}_{i}=f+\phi$ and $\mathcal{G}_{i}=g+\psi$ where $\phi$ and $\psi$ are chosen using Proposition \ref{perturb} for the interval $[a_{i},b_{i}]$ and $f$ and $g$ are the $C^m$ Whitney extensions of $F$ and $G$ respectively chosen earlier.
Define $\mathcal{H}_i$ to be the horizontal lift of $\mathcal{F}_i$ and $\mathcal{G}_i$
with starting height $H(a_i)$:
\begin{equation}
\label{Hlift}
\mathcal{H}_i(x) 
:= H(a_i) + 2\int_{a_i}^x (\mathcal{F}_i'\mathcal{G}_i - \mathcal{F}_i\mathcal{G}_i') 
\quad \text{for all } x \in [a_i,b_i].
\end{equation}
Differentiating $\mathcal{H}_i$ gives for any $x \in [a_i,b_i]$ and $1 \leq i \leq m$ the equation
\begin{equation}
\label{HorizH}
D^k\mathcal{H}_i(x) = \mathcal{P}^k(\mathcal{F}_i(x),\mathcal{G}_i(x),\mathcal{F}_i'(x),\mathcal{G}_i'(x),\dots,D^k \mathcal{F}_i(x),D^k\mathcal{G}_i(x))
\end{equation}
by the definition of the polynomials $\mathcal{P}^k$.


Clearly $(\mathcal{F}_{i}, \mathcal{G}_{i}, \mathcal{H}_{i})$ is horizontal by definition of the horizontal lift. It is of class $C^m$ because $f,g,\phi,\psi$ are at least $C^m$. 

Fix $1 \leq k \leq m$.
For $D^k \mathcal{F}_{i}(a_{i})$ we observe
\[D^k \mathcal{F}_{i}(a_{i})=D^k f(a_{i})+D^k \phi(a_{i})=F^{k}(a_{i})+0=F^{k}(a_{i})\]
by definition of $f$ and $\phi$. The same argument works for $D^k \mathcal{F}_{i}(b_{i})$, $D^k \mathcal{G}_{i}(a_{i})$, and $D^k \mathcal{G}_{i}(b_{i})$. 
For $D^k \mathcal{H}_{i}(a_{i})$ we calculate as follows using the assumption \eqref{HorizAssume}:
\begin{align*}
D^k \mathcal{H}_i(a_{i})&=\mathcal{P}^k(\mathcal{F}_i(a_i),\mathcal{G}_i(a_i),\mathcal{F}_i'(a_i),\mathcal{G}_i'(a_i),\dots,D^k \mathcal{F}_i(a_i),D^k \mathcal{G}_i(a_i))\\
&=\mathcal{P}^k( F(a_i),G(a_i),F^1(a_i),G^1(a_i),\dots,F^{k}(a_i),G^{k}(a_i))\\
&=H^{k}(a_{i}).
\end{align*}
The same argument shows $D^k \mathcal{H}_{i}(b_{i})=H^{k}(b_{i})$. 
Finally,
$\mathcal{H}_i(a_i) = H(a_i)$ by definition, 
and Proposition~\ref{perturb}(3) together with \eqref{Hlift}
gives $\mathcal{H}_i(b_i) = H(b_i)$.
This proves (1).

For (2) we use the definition of $\phi, \psi$ and \eqref{modcty} to estimate for $t\in [a_{i},b_{i}]$:
\begin{align*}
|D^k \mathcal{F}_{i}(x)-F^{k}(a_{i})| 
\leq |D^k \phi(x)| + |D^k f(x)-D^k f(a_{i})|
&\leq \beta(b_{i}-a_i)+\alpha(b_{i}-a_{i})\\
&\leq 2\beta(b_{i}-a_i).
\end{align*}
The same argument also yields $|D^k \mathcal{G}_{i}(x)-G^{k}(a_{i})|\leq 2\beta(b_{i}-a_i)$ which concludes the proof.
\end{proof}

Recall the interval $I=[\min K,\ \max K]$ and the decomposition $I\setminus K = \cup_{i\geq 1} (a_{i},b_{i})$, where the intervals $(a_{i},b_{i})$ are disjoint and $a_{i},b_{i}\in K$.

\begin{proposition}\label{conclusion}
Define the curve $\Gamma = (\mathcal{F},\mathcal{G},\mathcal{H}) \colon I\to \mathbb{H}^{1}$ as follows:
\[ \Gamma(x) := (F(x),G(x),H(x)) \quad \mbox{if }x \in K\]
and
\[\Gamma(x) := (\mathcal{F}_i(x),\mathcal{G}_i(x),\mathcal{H}_i(x)) \quad \mbox{if }x \in (a_i,b_i) \mbox{ for some }i \in \mathbb{N}.\]
Then $\Gamma$ is a $C^{m}$ horizontal curve in $\bbH^{1}$ with
\[D^k \mathcal{F}(x)=F^{k}(x),\qquad D^k \mathcal{G}(x)=G^{k}(x), \qquad D^k \mathcal{H}(x)=H^{k}(x)\]
 for all $x\in K$ and $0\leq k\leq m$.
\end{proposition}

\begin{proof}
Clearly the curve $\Gamma$ is $C^{m}$ in the subintervals $(a_{i}, b_{i})$. Define maps $\gamma^k$ on $K$ for $0 \leq k \leq m$ by $\gamma^k=(F^k,G^k,H^k)$. 
With this notation we have to show that $\Gamma$ is a $C^m$ horizontal curve and $D^k \Gamma|_{K}=\gamma^{k}$ for $0\leq k\leq m$.

Fix $k \in \{1,\dots,m\}$ and suppose we have shown that $D^{k-1} \Gamma$ exists on $I$ and $D^{k-1} \Gamma|_K = \gamma^{k-1}$.
(In the case $k=1$, this follows from the definition of $\Gamma$.)
In Lemma~\ref{buildcurve}, we showed that a one-sided derivative $D^k \Gamma$ exists from the right at $a_i$ and from the left at $b_i$ 
for every $i \in \mathbb{N}$ and takes the correct value.
Fix $x \in K$ with $x \neq a_i$ for any $i \in \mathbb{N}$ and $x \neq \max  K $. We now prove differentiability of $D^{k-1} \Gamma$ at $x$ from the right.
Suppose $\{x_i\}$ is a decreasing sequence converging to $x$.
We will show that 
\begin{equation}
\label{diffquot}
(x_i - x)^{-1} | D^{k-1} \Gamma(x_i) - D^{k-1} \Gamma(x) - (x_i-x)\gamma^k(x) |
\end{equation}
vanishes as $i \to \infty$.
If $\{x_i\} \subset K$, then \eqref{diffquot} indeed vanishes since $D^{k-1} \Gamma|_K = \gamma^{k-1}$ and $\gamma$ is a triple of Whitney fields on $K$.
If $\{x_i\} \subset I \setminus K$, 
then there exist $j_i\in \bbN$ for every $i\in \bbN$ such that $x_i \in (a_{j_i},b_{j_i})$.
For each $i \in \mathbb{N}$, we may bound \eqref{diffquot} by
\begin{align}
(x_i - x)^{-1} 
| &D^{k-1} \Gamma(x_i) - D^{k-1} \Gamma(a_{j_i}) - (x_i-a_{j_i})\gamma^k(a_{j_i}) | \label{Bound1} \\
&+
(x_i - x)^{-1} 
| (x_i - a_{j_i}) \gamma^k(a_{j_i}) - (x_i-a_{j_i})\gamma^k(x) | \label{Bound2} \\
&+
(x_i - x)^{-1} 
| \gamma^{k-1}(a_{j_i}) - \gamma^{k-1}(x) - (a_{j_i} - x)\gamma^k(x) |. \label{Bound3} 
\end{align}
Since $(x_i - a_{j_i}) < (x_i - x)$,
\eqref{Bound1} may be bounded by 
$$
\frac{1}{x_i - a_{j_i}} \int_{a_{j_i}}^{x_i} \left|D^k \Gamma(t) - \gamma^k(a_{j_i}) \right| \, dt
\leq 
\sup_{t \in (a_{j_i},b_{j_i})} \left|D^k \Gamma(t) - \gamma^k(a_{j_i}) \right|.
$$
Since
\begin{align*}
&|D^k \mathcal{H}(t) - H^k(a_{j_i})| \\
&\quad =
\Big| 
\mathcal{P}^k \left(\mathcal{F}(t),\mathcal{G}(t),\mathcal{F}'(t),\mathcal{G}'(t),\dots,D^k \mathcal{F}(t),D^k \mathcal{G}(t)\right)  \\
& \hspace{1in} - \mathcal{P}^k \left(F(a_{j_i}),G(a_{j_i}),F'(a_{j_i}),G'(a_{j_i}),\dots,F^{k}(a_{j_i}),G^{k}(a_{j_i})\right) 
\Big| 
\end{align*}
for any $t \in (a_{j_i},b_{j_i})$
and since each $\mathcal{P}^k$ is a polynomial,
property (2) in Lemma~\ref{buildcurve}
implies that \eqref{Bound1} vanishes as $(b_{j_i} - a_{j_i}) \to 0$.
The term \eqref{Bound2} vanishes uniformly as well
since $\gamma^k$ is continuous on the compact set $K$,
and \eqref{Bound3} also vanishes
due to the fact that $\gamma$ is a triple of Whitney fields on $K$. 
We have therefore shown that \eqref{diffquot} vanishes if $\{x_i\}$ is an arbitrary decreasing sequence converging to $x$.
That is, the right-hand derivative of $D^{k-1} \Gamma$ at $x$ is $\gamma^k(x)$.
We may argue similarly using increasing sequences
to show that the left-hand derivative of $D^{k-1} \Gamma$ at any $x \in K$ exists and is equal to $\gamma^k(x)$
(since Lemma~\ref{buildcurve}(2) implies that $\left|D^k \mathcal{F}_{j_i} - F^k(b_{j_i}) \right|$
and $\left|D^k \mathcal{G}_{j_i} - G^k(b_{j_i}) \right|$ will vanish as $(b_{j_i}-a_{j_i}) \to 0$).
That is, $D^k \Gamma$ exists on $I$, and $D^k \Gamma|_K = \gamma^k$.

It remains to prove that $D^m \Gamma$ is continuous on $I$.
The right and left-hand continuity of $D^m \Gamma$ at $a_i$ and $b_i$ respectively 
follow from the fact that $(\mathcal{F}_{i}, \mathcal{G}_{i}, \mathcal{H}_{i})$ is $C^m$ on $[a_i,b_i]$.
If $x \in K$ and $x \neq a_i$, we may prove the right-hand continuity of $D^m \Gamma$ 
by taking a decreasing sequence $\{x_i\}$ as before.
Indeed, if $\{x_i\} \subset K$, the continuity of $\gamma^m$ on $K$ gives $D^m\Gamma(x_k) \to D^m \Gamma(x)$,
and, if $\{x_i\} \subset I \setminus K$, then writing
$$
|D^m \Gamma(x_k) - D^m \Gamma(x)| 
\leq 
|D^m \Gamma(x_k) - \gamma^{m}(a_{j_i})| 
+
|\gamma^{m}(a_{j_i}) - \gamma^{m}(x)| 
$$
shows the desired convergence
as in the proof of \eqref{diffquot}.
Arguing in the same way using increasing sequences as before gives the continuity from the left.
Hence $D^m \Gamma$ is continuous on $K$, and therefore $\Gamma$ is indeed $C^m$ on $I$.

Finally, note that $\Gamma$ is horizontal by the hypothesis \eqref{HorizAssume} on $(F,G,H)$ and the fact that 
$(\mathcal{F}_{i}, \mathcal{G}_{i}, \mathcal{H}_{i})$ is horizontal on each subinterval $(a_i,b_i)$.
\end{proof}

This proves Theorem \ref{newthm}. 
\end{proof}

Taken together, Proposition \ref{nec} and Theorem \ref{newthm} prove Theorem \ref{iff}, which is our main result.

\end{document}